\documentclass[12pt]{amsart}

\usepackage{amsmath,amsfonts,amsthm,amsopn,cite,mathrsfs}
\usepackage{epsfig,verbatim}
\usepackage{subfigure}
\usepackage[mathscr]{eucal}

\newcommand{\N}{\mathbb{N}}

\newcommand{\eps}{\varepsilon}
\newcommand{\A}{\mathcal{A}}
\newcommand{\B}{\mathcal{B}}
\newcommand{\ld}{L^2(\mathbb{R}^{d})}
\newcommand{\ldd}{L^2(\mathbb{R}^{2d})}
\newcommand{\rd}{\mathbb{R}^d}
\newcommand{\rdd}{\mathbb{R}^{2d}}
\newcommand{\li}{L^{\infty}}

\newcommand{\swrz}{\mathcal{S}}
\newcommand{\temp}{\mathcal{S}^{\prime}}
\newcommand{\loc}{A_a^{\varphi_1,\varphi_2}}
\newcommand{\sch}{\mathscr{S}}
\newcommand{\tc}{\mathscr{S}^1}
\newcommand{\hs}{\mathscr{S}^2}

\newtheorem{lemma}{Lemma}[section]
\newtheorem{theorem}[lemma]{Theorem}
\newtheorem{definition}[lemma]{Definition}
\newtheorem{corollary}[lemma]{Corollary}
\newtheorem{proposition}[lemma]{Proposition}

\newcommand{\tfa}{time-frequency analysis}

\newcommand{\stft}{short-time Fourier transform}

\newcommand{\tf}{time-frequency}

\newcommand{\fif}{if and only if}
\newcommand{\tfs}{time-frequency shift}

\newcommand{\modsp}{modulation space}
\newcommand{\psdo}{pseudodifferential operator}

\newcommand{\rems}{\noindent\textsl{REMARKS:}}
\newcommand{\rem}{\noindent\textsl{REMARK:}}

\newcommand{\beqa}{\begin{eqnarray*}}
\newcommand{\eeqa}{\end{eqnarray*}}

\DeclareMathOperator*{\supp}{supp}

\newcommand{\field}[1]{\mathbb{#1}}
\newcommand{\bR}{\field{R}}        
\newcommand{\bN}{\field{N}}        
\newcommand{\vf}{\varphi}

 \def\cF{\mathcal{F}}              
 \def\cS{\mathcal{S}}

 \def\cB{\mathcal{B}}

 \def\cA{\mathcal{A}}
 \def\cJ{\mathcal{J}}

\def\rd{\bR^d}

\def\rdd{{\bR^{2d}}}

\def\lrd{L^2(\rd)}
\def\lrdd{L^2(\rdd)}

\def\mif{M^{\infty,1}}

\def\intrdd{\int_{\rdd}}

\def\<{\left<}
\def\>{\right>}

\def\mv1{M_v^1}

\newcommand{\vs}{\vspace{3 mm}}

\begin{document}


\begin{abstract}
Time-frequency localization operators are a  quantization procedure  that
maps symbols on $\rdd $ to operators and depends on two window
functions. We study the range of this quantization and its dependence
on the window functions. If the \stft\ of the windows does not have
any zero, then the range is dense in the Schatten $p$-classes. The
main tool is new version  of the  Berezin transform associated to 
operators on $\lrd $. Although some results are analogous to results about
Toeplitz operators on spaces of holomorphic functions,  the
absence of a complex structure requires the  development of  new methods that
are based on \tfa . 
\end{abstract}


\title{Time-Frequency  Localization Operators and a Berezin Transform}
\author{Dominik Bayer}
\address{Acoustics Research Institute \\
Austrian Academy of Sciences \\
Wohllebengasse 12-14 \\
A-1040 Vienna, Austria}
\email{dominik.bayer@kfs.oeaw.ac.at}
\author{Karlheinz Gr{\"o}chenig}
\address{Faculty of Mathematics \\
University of Vienna \\
Oskar-Morgenstern-Platz 1\\
A-1090 Vienna, Austria}
\email{karlheinz.groechenig@univie.ac.at}
\date{}
\thanks{K.\ G.\ was
  supported in part by the  project P26273 - N25  of the
Austrian Science Fund (FWF)}
\maketitle


\section{Introduction}

We study the problem whether an arbitrary  operator on $\lrd $ can be
approximated by a \tf\ localization operator. The formalism of \tf\
localization operators is  a  quantization procedure  and 
maps functions on phase space $\rdd $ to operators acting on $\lrd
$. Time-frequency localization operators are widely used in physics
and engineering. In physics, a special case of localization  operators
(with Gaussian window) has been around for
a long time in connection with quantization under the name
(Anti-) Wick operators in the work of Berezin~\cite{be71}. In analysis,  they are used to approximate
\psdo s, see \cite{cofe78,lerner03,sh01}.  
 In the form considered here localization operators  were 
introduced and studied by  Daubechies~\cite{da88} and Ramanathan and
Topiwala~\cite{rato93}. In signal processing localization operators
are  used for \tf\ masking and  feature extraction of signals    from
a  time-frequency representation.

For a more formal discussion of our  results, let us introduce the \tfs s
$\pi (z) f(t) = e^{2\pi i \xi \cdot t} f(t-x)$ for $z=(x,\xi ) \in
\rdd $ and $t\in \rd $. If $a$ is a function on $\rdd $ and $\varphi _1,
\varphi _2$ are two  functions in $\lrd $, so-called  window
functions,  then the
localization operator $A_a^{\varphi _1,\varphi _2}$ is defined formally as 
\begin{align*}
A_a^{\varphi _1, \varphi _2} f = \intrdd a(z)  \langle f, \pi (z) \varphi
_1\rangle \pi  (z) \varphi _2 \, dz \, .
\end{align*}

Roughly speaking, the action of $A_a^{\varphi _1, \varphi _2}$ consists of
multiplying the \stft\ $z \mapsto \langle f, \pi (z) \varphi _1\rangle$ by a
symbol $a$ (this yields a function on $\rdd $) and then projecting
back to $\lrd $. This procedure resembles the definition of Toeplitz
operators in complex analysis, and indeed, localization operators are
also  called Toeplitz operators by some authors~\cite{Toft04,toft08a}.  
The basic properties of localization operators are well
understood. There  are numerous results about  the boundedness properties between function spaces, the
Schatten class properties, and the  symbolic calculus for localization
operators. The reader should consult the surveys~\cite{cogrro06,lerner03,wong02} for the
basic properties
and~\cite{BT05,cogr03,cogr06,FG06,GT11} for further
discussion.

In this paper we  examine the
question how large the class of localization operators (with symbols
from some prescribed class, e.g., from $\lrd$) is compared to the
standard spaces of operators (e.g., the Hilbert-Schmidt class). Can an  arbitrary
operator be approximated by a  localization operator,  and, if yes,
in which topology? In terms of physics, what is the range of the quantization
procedure given by localization operators?

Motivated by the analogy to Toeplitz operators, we will answer this
question in terms of a 
related dual  mapping, the so-called Berezin transform. Fix
$\varphi_1, \varphi_2 \in \lrd$, and 
let $T\in B(\lrd)$ be a bounded operator on $\lrd $.
  The Berezin transform maps operators to functions and is defined to be 
\[ \mathcal{B}T(z) := \left<T\pi(z)\varphi_2, \pi(z)\varphi_1
\right>, \quad\quad z\in\rdd.
\]

 In the following $\sch ^p$ denotes the Schatten $p$-class on
$\lrd $ for $1\leq p <\infty$.  For $p= \infty $, $\sch ^\infty $ consists of all compact
operators on $\lrd $. The conjugate index is $p'=p/(p-1)$. 
Time-frequency  localization operators and the Berezin
transform are related as follows. 

\begin{theorem}\label{mainsimple}
Let $\varphi _1, \varphi _2\in \lrd $, $1\leq    p^{\prime} < \infty$. 
 Then the range $\{A_a^{\varphi _1, \varphi _2}: a\in L^{p'}(\rdd )\} $
is norm-dense in $\sch^{p^{\prime}}(\lrd)$, \fif\ the Berezin
transform $\mathcal{B}$ is one-to-one on $\sch^p$. 
\end{theorem}

In other words, the quantization procedure $a \mapsto \loc $ has dense
range in the Schatten class $\sch ^{p'}$, \fif\ the Berezin transform
is one-to-one on $\sch ^p$. This insight leads us to investigate the
properties of the Berezin transform.

Here is a simple  statement about the  injectivity of the Berezin
transform without any additional technical definitions. 

\begin{proposition} \label{inj12}
Let $\varphi_1, \varphi_2 \in \ld $. 

(i)  If $1\leq p\leq2$ and $V_{\varphi_1} \varphi_2 (z) \not= 0$ for almost
all $z\in \rdd$, then $\B$ is one-to-one from $\sch^p$ to $L^p(\rdd)$.  In
this case, $ \{ \loc : a \in L^{p'}(\rdd )\}$
is norm-dense in $\sch ^{p'}$. 

(ii)  If $p\geq 2$ and $\B$ is one-to-one from $\sch^p$ to $L^p(\rdd)$,
then $V_{\varphi_1} \varphi_2 (z)  \not= 0$ for almost all $z\in \rdd$. 
\end{proposition}

The main point of this proposition is that the injectivity  of the Berezin
transform depends on the properties of the windows $\vf _1$ and $\vf
_2$. 

By combining Theorem~\ref{mainsimple} and Proposition~\ref{inj12} we
obtain a complete characterization when the quantization $a \mapsto
\loc $ has dense range in the class of Hilbert-Schmidt operators $\sch ^2$. 
\begin{corollary}
The range   $\{A_a^{\varphi
    _1, \varphi _2} : a\in L^{p'}(\rdd )\} $ is norm-dense in $\sch
^2$, \fif\ $\langle \varphi _2 , \pi (z) \varphi _1\rangle \neq 0$  for
almost  all $z\in\rdd$. 
\end{corollary}

For $p=\infty $, we will see that the range of $a\mapsto \loc $ from
$L^\infty (\rdd )$ to $B(\lrd )$ is dense in the weak operator
topology, but fails to be  norm-dense. In fact, for the Fourier transform $\cF $ we will show
(Theorem~\ref{T:l_infty_counterexample_fourier_transform}) that 
$$
\|\cF - A_a^{\varphi _1, \varphi _2}\|_{op} \geq 1 
$$
for all $a\in L^\infty (\rdd )$ and all $\varphi _1,\varphi _2\in \lrd
$. Therefore the class of localization operators cannot be norm-dense
in $B(\lrd )$.

 Theorem~\ref{mainsimple} and Proposition~\ref{inj12} are formulated  for $L^p $-symbols.
 For such symbols the corresponding mapping properties between symbols and
operators are easy to understand~\cite{bocogr04,wong02}. However,  in the context of time-frequency analysis 
a less restrictive calculus is available that also includes 
distributional symbols. In Sections~4 and~5   we will formulate more
technical  versions of Theorem~\ref{mainsimple} and Proposition~\ref{inj12} for distributional symbol
 classes.

 Theorem~\ref{mainsimple} and Proposition~\ref{inj12}, at least for the cases
$p=1$ and $p=\infty $,  have many predecessors  in complex analysis,
see, e.g., ~\cite{bergercoburn87,bergercoburn94,Eng92,Eng96,Eng06}. Perhaps the closest
results are due to Berger and Coburn who show that  Toeplitz operators on
Bargmann-Fock space with symbols of compact support are dense in the
trace class and in the compact operators~\cite[Thm.~9]{bergercoburn94}. 

Finally, Theorem~\ref{mainsimple}  can also  be formulated in a  more
abstract context, where one needs only a resolution of the identity
$\mathrm{Id} = \int _X P(x) d\mu (x)$ with rank-one operators
$P(x)$. This formalism was introduced by B.\ Simon~\cite{Sim80} and some
special cases of Theorem~\ref{mainsimple} could be attributed to him.
We emphasize that in the context of \tfa\ we prove a much sharper
result  for distributional symbol classes that cannot be formulated in
the axiomatic setting.

Perhaps the main surprise of
Theorem~\ref{mainsimple} is the explicit dependance of the density
property on the properties  of the window functions. This is a new
aspect in our approach, because  window
functions do not even occur in the definition of classical  Toeplitz operators. The
analysis of the interaction between window functions and the
properties of \tf\ localization operators is the main novelty of this
paper and requires tools that are specific for \tfa .

The paper is organized as follows:   In Section~2 we collect the
necessary prerequisites from \tfa , in particular, we recall the
definition of \modsp s and some facts from harmonic analysis, and   we  summarize the
mapping properties of localization operators. In Section~3 we
introduce the Berezin transform and study its mapping properties. In
Section~4 we investigate the kernel of the Berezin transform.  Finally,  in
Section~5 we  study the approximation of arbitrary  operators  by \tf\
localization operators and prove the  theorems discussed in the
introduction. 


\section{Prerequisites}


\subsection{Short-Time Fourier Transform, Wigner Distribution and Modulation Spaces}

Let $x,\omega \in\rd$ and $z = (x,\omega)\in\rdd$. We define the
translation operator on $\ld$ as $T_xf(t) := f(t-x)$ 
and the modulation operator as
$M_\omega f(t) := e^{2\pi i \omega\cdot t}f(t)$, for $f \in \ld$.
The time-frequency shift $\pi(z)$ is then given as
\[
\pi(z)f(t) :=
\pi(x,\omega)f(t) := M_\omega T_xf(t) = e^{2\pi i \omega \cdot t}f(t-x).
\]
The \tfs s  are unitary operators on the Hilbert space $\ld$ and
isometries on $L^p(\rd )$. 

We frequently use the following bilinear time-frequency distributions.

The short-time Fourier transform of $f$ with respect to a  window $g$ (STFT) is defined as
\[
V_gf(x,\omega) := \left< f, M_{\omega}T_xg \right> = \int_{\rd} f(t)\overline{g(t - x)}e^{-2\pi i \omega \cdot t}\,dt
\]
for $x,\omega \in \rd$.

We list some standard properties of the \stft .

\begin{lemma} \label{l:prop}
  Let $f,g\in \lrd $. Then the following holds: \\
(i) Covariance property: $|V_g(\pi (w)f)(z)| = |V_gf(z-w)|$ \\
(ii) Isometry property: $V_gf $ is in $\lrdd $ and $\|V_gf\|_2 =
\|f\|_2 \|g\|_2$. \\
(iii) Let $\cJ (z_1,z_2) = (z_2, -z_1)$ be the standard symplectic form on
$\rdd $, then  $$
 V_{\pi (z) g} (\pi (z) f)(w)  =   e^{2\pi i \cJ z
  \cdot w} V_{g} f(w) \, \qquad w,z \in \rdd \, .
$$ 
\end{lemma}

For the representation of operators we will also need the Wigner
distribution. 
 The (cross) Wigner distribution of $f\in \ld $ and $g\in \ld
 $ is defined to be
\[
W(f,g)(x,\omega) := \int_{\rd} f(x + \tfrac{t}{2})\overline{g(x - \tfrac{t}{2})}e^{-2\pi i \omega \cdot t}\,dt
\]
for $x, \omega \in \rd$.

The following formula connects the Wigner distribution with the STFT
via the Fourier transform. 
Let $f,g \in \ld$. Then 
\begin{align} \label{L:ftwig}
\widehat{W(f,g)}(x,\omega ) = e^{-\pi i x \cdot \omega} \, 
V_gf(-\omega, x), \quad\quad \mbox{for all } x,\omega \in \rd.
\end{align}

\begin{definition}[Modulation Spaces]\label{D:mod_sp}
Let $g\in\swrz(\rd)$ be a  fixed non-zero Schwartz function, and $1\leq p,q\leq \infty$.
We define the modulation space
\[ M^{p,q}(\rd) := \{ f\in\temp(\rd):
V_gf  \in L^{p,q}(\rdd)\}.\]
The norm on $M^{p,q}$ is given by 
$$
\|f\|_{M^{p,q}} = \|V_gf \|_{L^{p,q}} := \left(
\int_{\rd}\left(\int_{\rd}|V_gf (x,\omega )|^p\,dx
\right)^{q/p}\,d\omega \right)^{1/q}
$$
with modifications if $p=\infty $ or $q=\infty $. We write $M^p =
M^{p,p}$.  
\end{definition}

The theory of modulation spaces is explained in detail in the
monographs \cite{gr01} and \cite{fest98}.
 We mention only the following properties:
Every   modulation space $M^{p,q}$ is a  Banach space. Its  definition does not
 depend on the chosen window function $g$; different windows $g\in \cS
 (\rd )$ yield  equivalent norms~\cite{gr01}.
The duality of \modsp s is  analogous to that of  Lebesgue spaces: If  $1\leq p,q < \infty$,  then the  dual space of $M^{p,q}$
is given by $M^{p',q'}$, where $p' = \frac{p}{p-1}$ is the conjugate
index of $p$. 
Unlike the Lebesgue spaces $L^p(\rd)$, modulation spaces are embedded into each other: $\mathcal{S}\subseteq M^1 \subseteq M^p \subseteq M^q \subseteq M^\infty
\subseteq \mathcal{S}' $ for $1 \leq p \leq q \leq
\infty$. Furthermore, 
$$
M^{p,1} \subseteq L^p \subseteq
M^{p,\infty}$$
 for all $1 \leq p \leq \infty$.\\

For further reference we note the following smoothness properties of
STFTs~\cite{gr01,cogr03}. If $f,g\in \cS (\rd )$, then $V_gf
\in \cS (\rdd )$; if $f,g \in M^1(\rd )$, then $V_gf \in M^1(\rdd
)$.  \\
Finally, the we will use the  following
convolution~\cite{cogr03,Toft04}: $M^\infty \ast M^1 \subseteq \mif $
with a norm estimate 
\begin{equation}
  \label{eq:c1}
\|f \ast g \|_{\mif } \leq C \|f\|_{M^\infty }  \, 
\|g\|_{M^1} \, .  
\end{equation}


\subsection{Compact Operators and Schatten Classes}

Let $T:H\to H$ be a compact and self-adjoint operator on a Hilbert
space $H$. Then there exists a 
sequence of eigenvalues $(\lambda_j)_{j\in\N}\subseteq \bR $  and an orthonormal
set of eigenvectors  $(\phi_j)_{j\in\N} \subseteq H$ such that
$\lim_{j\to\infty} \lambda_j = 0$, $T\phi_j = \lambda_j\phi_j$ for all
$j\in\N$, 
and
\[
Tf = \sum_{j=1}^{\infty}\lambda_j\left<
f,\phi_j\right>\phi_j
\]
for all $f\in H$, with convergence of the
series in the norm of $H$.
 The singular values $s_j(T)$ of a compact operator $T:H\to H$
are the square-roots of the eigenvalues of the (compact
self-adjoint positive) operator $T^*T$. We write
\[ s_j(T) := \lambda_j(T^*T)^{1/2}, \]
 and assume that $s_j(T) \geq s_{j+1}(T)$. 
If $(s_j(T))_{j\in\N} \in \ell^p(\N)$ for some $1 \leq p < \infty$,
then  $T:H\to H$ belongs to the Schatten 
$p$-class $\mathscr{S}^p$.   Equipped with the norm $\| T
\|_{\mathscr{S}^p} = \big(\sum_j s_j(T)^p \big)^{1/p}$, $\mathscr{S}^p$ is a Banach space for $1 \leq p
< \infty$. The space  $\mathscr{S}^1$ consists of the  trace class
operators, the space  $\mathscr{S}^2$ consists of all Hilbert-Schmidt
operators. For $p = \infty$, we define $\cS^{\infty}$ to be the set of
all compact operators. Then we have the duality relation
$(\mathscr{S}^{p})^{\ast} = \mathscr{S}^{p^{\prime}}$, where $p' =
\frac{p}{p-1}$ is the conjugate 
index of $p$, $1 \leq p < \infty$.

\subsection{Localization Operators}

 A \tf\ localization operator  can be interpreted as  a multiplier
 for the short-time Fourier transform. First we analyze 
a given function $f$  on $\rd$, a "signal" in engineering language,
by  taking its short-time Fourier transform with window
$\varphi_1$. This yields a time-frequency representation of $f$ on
$\rd \times \rd$. Then we  extract the  interesting parts of this 
representation by
multiplying the STFT  with a suitable "mask", the so-called symbol, a function
$a$ on $\rd \times \rd$. Finally, we  synthesize a new signal by
applying the adjoint of the  short-time Fourier transform (possibly with some other window
$\varphi_2$). 
This algorithm of time-frequency filtering leads to the formal
definition of the localization operator:
\begin{align*}
A_a ^{\vf _1, \vf _2} f =  V_{\varphi_2}^*(a\cdot V_{\varphi_1}f) =
\int_{\rdd} a(z)\cdot
V_{\varphi_1}f(z)\,\pi (z) \varphi _2  \, dz \, . 
\end{align*}
In proofs it will be convenient to use the weak definition of a
localization operator. If $\varphi _1, \varphi _2 , f,g \in \lrd $, then 
\begin{equation}
  \label{eq:c2}
  \langle \loc f  , g \rangle = \langle a , \overline{V_{\varphi _1}f}
  \, 
  V_{\varphi _2}g \rangle \, . 
\end{equation}
Here the first bracket is  the inner product on $\lrd $,
whereas  the second bracket is on $\lrdd $.  By extending the inner
product to a duality bracket, one can define  a
localization operator conveniently on many other function spaces. For
example, \eqref{eq:c2} makes sense, when $a \in \cS ' (\rdd )$ and
$\varphi _1, \varphi _2, f,g \in \cS (\rd )$, consequently, the
localization operator 
$\loc $ 
defines a continuous operator $\loc $ from $\cS (\rd )$ to $\cS '
(\rd )$. As a special case  we mention that  for  $a \in M^\infty (\rdd )$ and $\varphi _1, \varphi
_2\in M^1(\rd )$ the operator $\loc $ is bounded on all \modsp s
$M^{p,q}(\rd ), 1\leq p,q \leq \infty$~\cite{cogr03}.

The following table summarizes the basic  mapping properties of
the correspondence  $\mathcal{A}:a \mapsto \loc$ from symbols  to
localization operators. In the result for compact operators $L^0(\rdd
)$ denotes the subspace of compactly supported functions in $L^\infty
(\rdd )$. For proofs see the given references.

\begin{table}[h]
\begin{center}
\begin{tabular}{| c | c | c | c |}
\hline
\textbf{Symbol} & \textbf{Windows} & \textbf{Localization Operator} &  \textbf{Ref.} \\ \hline\hline
$\temp(\rdd)$ & $\swrz(\rd)$ & $\Psi$DO $\swrz(\rd) \to \temp(\rd)$ & \cite{da88} \\ \hline
$L^\infty(\rdd)$ & $\ld$ & $B(\ld)$ & \cite{wong02} \\ \hline
$L^0(\rdd)$ & $\ld$ & $\sch^\infty(\ld)$ & \cite{FG06} \\ \hline
$L^{p}(\rdd)$, $1\le p < \infty$ & $\ld$ & $\sch^p(\ld)$ & \cite{bocogr04} \\ \hline
$M^{\infty,\infty}(\rdd)$ & $M^1(\rd)$ & $B(M^{p,q}(\rd))$, $1\le p,q
\le \infty$ & \cite{cogr03}  \\ \hline 
$M^{p,\infty}(\rdd)$, $1\le p < \infty$ & $M^1(\rd)$ & $\sch^p(\ld)$ &
\cite{cogr03} \\ \hline
\end{tabular}
\end{center}
\caption{Localization operators with different symbols and windows}
\end{table}

We mention that the symbol class $M^{p,\infty }$ is significantly
larger than just $L^p$. For example,  the \modsp\ $M^{p,\infty }$
contains discrete measures of the form $a = \sum _j c_j \delta
_{z_j}$, 
where the  coefficient sequence  $c$ is in  $ \ell ^p$ and $\{ z_j \}
$ is a uniformly discrete subset of $\rdd $. In particular, with the
symbol class $M^{p,\infty }$ one can treat many aspects of the theory
of so-called Gabor multipliers that was initiated in~\cite{FN03}.


\subsection{Kernel Theorems}

For some technical arguments we will  use the \psdo\
representation  of an operator. We state explicitly two representation
theorems that use the Wigner distribution $W(f,g)$   of $f$ and
$g$.

\begin{proposition}\label{T:kernel_hs}
(i) Pool's Theorem \cite{po66}: If $T\in\mathscr{S}^2(\ld)$,  then
there exists a unique symbol $\sigma\in\ldd$ such that  \[ \left<Tf,g\right> = \left<\sigma, W(g,f)\right>
\]
for all $f,g \in\ld$.

(ii) \label{T:kernel_modulation} Kernel theorem for \modsp s~\cite{fegr97,gr01}: Let $T:M^1(\rd)\to M^{\infty}(\rd)$ be a bounded linear operator. Then there exists a unique $\sigma \in M^{\infty}(\rdd)$ such that
\begin{align*}
\left< Tf, g \right> = \left< \sigma, W(g,f) \right>
\end{align*}
for all $f,g \in M^1(\rd)$. 
\end{proposition}


\subsection{Density theorems}

Tauberian theorems address the question when the span of all
translates of a given function $f$ is dense in a space. It is
therefore not surprising that they can be applied with profit to other
density questions. We need three versions. 

\begin{proposition}
\label{T:tauber_l2} \label{T:tauber_classic} 
(i) $L^2$-version: Let $f\in\ld$. Then $\overline{\text{span}\, \{T_zf|\; z\in\rd\}} = \ld$
if and only if $\hat{f}(\omega) \not= 0$ for almost all
$\omega\in\rd$.

(ii) $L^1$-version: Let $f\in L^1(\rd)$. Then
$\overline{\text{span}\, \{T_zf|\; z\in\rd\}} = L^1(\rd)$, 
if and only if $\hat{f}(\omega) \not= 0$ for all
$\omega\in\rd$.

(iii) \label{T:tauber_m1} $M^1$-version: Let $f\in M^1(\rd)$. Then 
$\overline{\text{span} \, \{T_zf|\; z\in\rd\} } =  M^1(\rd)$  (where
the closure is taken with respect to the
$M^1$-norm),  
if and only if $\hat{f}(\omega) \not= 0$ for all
$\omega\in\rd$.
\end{proposition}

The $L^1$ and $L^2$-versions are  standard, the $L^1$-version is one of
many famous theorems of Wiener~\cite{rest00}, the $M^1$-version is
contained in \cite{fe73,re71,rest00} (the ideal theory in
the Banach algebras $L^1$ and $M^1$ is the same). 

Let us emphasize that the obvious generalizations to $L^p(\rd )$ for $ p\neq
1,2,\infty $ are false. By a very deep result of Lev and Olevski~\cite{LO11} the density of $\overline{\text{span}\, \{T_zf|\;
  z\in\rd\}} $ in $L^p$  depends on subtle properties of $f$, and not
only on the measure of the zero set of $\hat{f}$. For this reason the
formulation of Proposition~\ref{inj12} is divided into a sufficient
condition and a slightly weaker necessary condition (rather than a
complete characterization).


\section{The Berezin Transform}

In this section we study a general Berezin transform for operators on
$\lrd $. The Berezin transform maps operators to functions, whereas
quantization rules map functions to operators.

\begin{definition}[Berezin Transform]\label{D:berezin}
Let $T\in B(\ld)$. Let $\varphi_1, \varphi_2 \in \ld$. The  Berezin
transform  $\mathcal{B}$ maps
$T$ to the function on $\rdd$ given by 
\[ \mathcal{B}T(z) := \left<T\pi(z)\varphi_2, \pi(z)\varphi_1
\right>, \quad\quad z\in\rdd.
\]
\end{definition}

\rems\  1.  The Berezin transform and its properties  depend on the 
window functions $\varphi_1, \varphi_2 \in \ld$, but we will simply
write $\B$ instead of  $\B^{\varphi_1, \varphi_2}$.  This practice
should not lead to any confusion.

2. If $\varphi _1, \varphi _2 \in \cS (\rd )$, then the Berezin transform
can be defined for all operators from $T: \cS (\rd ) \to \cS ' (\rd
)$. In the following we will derive several continuity properties of
the Berezin transform for different spaces of operators and windows.

For technical  arguments  we will   use  two alternative
representations of  the  Berezin transform.

\begin{lemma} \label{lem:repber}
   If $T$ is bounded from $M^1(\rd )$ to $M^\infty (\rd )$ and
  $\varphi _1, \varphi _2 \in M^1(\rd )$, then there exists a unique  symbol $\sigma
  \in M^\infty (\rdd )$, such that 
  \begin{equation}
    \label{eq:c3}
    \B T(z) = \langle \sigma , T_z W(\varphi _1, \varphi _2)\rangle \, .
  \end{equation}
If  $\varphi _1, \varphi _2 \in \lrd $ and $T \in \sch ^2$, then
\eqref{eq:c3} holds with  $\sigma \in \lrdd $. 
\end{lemma}
\begin{proof}
Let $\sigma $ be the kernel of $T$ in the \psdo\ representation of $T$
given in Theorem~\ref{T:kernel_hs}. If $T : M^1 (\rd )  \to M^\infty
(\rd )  $, then $\sigma \in
M^\infty (\rdd )$; if $T\in \sch ^2$, then $\sigma \in \lrdd
$.   The identity~\eqref{eq:c3} now follows from the  identity
  \begin{align*}
BT(z) & = \langle T \pi (z) \varphi _2 , \pi (z) \varphi _1 \rangle  =
\langle \sigma , W(\pi (z)\varphi _1, \pi (z) \varphi _2)\rangle \\
&= \langle \sigma , T_z W(\varphi _1, \varphi _2) \rangle \, , 
  \end{align*}
where in the last equality we have used the covariance property of the
Wigner distribution. 
\end{proof}

\begin{lemma}\label{L:berezin_representation}
Assume that  $(g_n)_{n\in\N}$ and $(h_n)_{n\in\N}$ are two
orthonormal systems  in $\ld$ and  that $(s_n)_{n\in\N} \in \ell ^\infty$
is  a bounded sequence of complex numbers.   
If 
\begin{equation}
Tf = \sum_{n\in\N} s_n \left< f, g_n \right>h_n \, ,  
\end{equation}
then
\begin{equation}
\B T(z) = \sum_{n\in\N} s_n V_{\varphi_1}h_n(z) \overline{V_{\varphi_2}g_n(z)}
\end{equation}
for every $z\in\rdd$.
\end{lemma}
\begin{proof}
Since $(g_n)$ and $(h_n)$ are orthonormal systems in $\lrd $, the series defining
$T$ converges in $\ld$ for every $f$ and $T$ is bounded on $\lrd $. 
We compute
\begin{align*}
\B T(z) & = \left< T \pi(z)\varphi_2, \pi(z)\varphi_1 \right> \\
& = \left< \sum_{n\in\N} s_n \left< \pi(z)\varphi_2, g_n \right>h_n, \pi(z)\varphi_1 \right> \\
& = \sum_{n\in\N} s_n \left< \pi(z)\varphi_2, g_n  \right> \left< h_n, \pi(z)\varphi_1 \right>\\
& = \sum_{n\in\N} s_n \overline{V_{\varphi_2}g_n(z)} V_{\varphi_1}h_n(z).
\end{align*}
\end{proof}

In particular, the assumptions of Lemma \ref{L:berezin_representation}
are  satisfied for compact  operators  and for operators belonging to
some Schatten $p$-class $\sch^p$, $1\le p \leq  \infty$, since every
$T\in \sch ^p$  possesses a  singular value decomposition
\begin{equation}
T = \sum_{n\in\N} s_n \left< \cdot , g_n \right>h_n
\end{equation}
that satisfies the assumptions of Lemma~\ref{L:berezin_representation}. 

\subsection{Boundedness of the Berezin Transform}

In the following we study the mapping properties of the Berezin
transform on the Schatten $p$-classes. 

\begin{theorem}\label{T:berezin_s_p}
Let $1\le p \le \infty$ and $\varphi _1,\varphi _2\in \ld $. \\
The Berezin transform is  bounded from $\sch^p$ to $L^p(\rdd)$ 
with operator norm
\[
\|\B \|_{\sch^p \to L^p} \le \|\varphi_1 \|_2\, \|\varphi_2 \|_2.
\]
If $T\in \sch ^\infty $, then $\B T$ is in $C_0(\rdd )$, the
continuous functions vanishing at infinity. 
\end{theorem}

\begin{proof}
We prove the boundedness for $p=1$ and $p=\infty $ and then use
interpolation. 

(i) Case $p=\infty$: If $T\in B(\ld )$, then 
\begin{align*}
|\mathcal{B}T(z)| &= | \left<T\pi(z)\varphi_2, \pi(z)\varphi_1
\right>| \\
&\le \|T\|_{B(L^2)}\, \|\pi(z)\varphi_2\|_2\, \|
\pi(z)\varphi_1\|_2 \\
&= \|T\|_{B(L^2)}\, \|\varphi_2\|_2\, \|\varphi_1\|_2
\end{align*}
for all $z\in\rdd$, hence $\B$  is  bounded from  $B(L^2)$ to $\li$ with
\[
\| \B \|_{B(L^2)\to L^{\infty}} \le \|\varphi_1 \|_2\, \|\varphi_2 \|_2.
\]
The function $\B T$ is continuous since  the mapping $z\mapsto \pi(z)\varphi$ is continuous
from $\rdd$ to $\ld$ for arbitrary
$\varphi\in\ld$. 

Next assume that  $T$ is a compact operator 
and $\lim _{n\to \infty } |z_n | = \infty$ for a sequence  $(z_n
)_{n\in \bN } \subset 
\rdd $. 
Then, for arbitrary $\phi, \psi \in \ld$,
\[
\langle \psi, \pi(z_n)\phi \rangle = V_{\phi}\psi(z_n) \longrightarrow 0 = \langle \psi, 0 \rangle
\]
by the version of the Riemann-Lebesgue Lemma for the STFT. Thus
$\pi(z_n)\phi $  converges weakly to $0$. Since $T$ is compact, 
the sequence $T (\pi(z_n)\phi_1) $ converges to $0$  in the $L^2$-norm. But then
\[
\B T(z_n) = \langle T\pi(z_n)\phi_1, \pi(z_n)\phi_2 \rangle \longrightarrow 0
\]
for $n \to \infty$. Consequently, $\mathcal{B} T \in C_0(\rdd )$. 

(ii) Case $p=1$: 
If  $T\in\tc$ is  trace class,  then $T$ possesses
a spectral representation 
\[ Tf = \sum_k s_k \left<f,g_k \right> h_k, \quad\quad f\in\ld,
\]
with two  orthonormal systems $(g_k)_{k\in\N}, (h_k)_{k\in\N}$, and singular values $(s_k)_{k\in\N} \in \ell
^1$, such that   $\sum_{k\in\N} s_k = \|T\|_{\mathscr{S}^1} <
\infty$. Then by  Lemma \ref{L:berezin_representation}
\[
\B T (z) = \sum_k s_k V_{\varphi_1}h_k(z) \overline{V_{\varphi_2}g_k(z)}.
\]
Since  $V_{\varphi_1}h_k$ and $V_{\varphi_2}g_k$ are in $\ldd$
by~Lemma~\ref{l:prop}(ii),  the product
$V_{\varphi_1}h_k\overline{V_{\varphi_2}g_k}$ is in
$L^1(\rdd)$. Therefore the 
above series converges absolutely in $L^1(\rdd)$ with 
\begin{align*}
\|\mathcal{B}T\|_1 & =  \|\sum_k s_k
V_{\varphi_1}h_k \overline{V_{\varphi_2}g_k} \|_1 
 \leq \sum_k s_k \|V_{\varphi_1}h_k
\overline{V_{\varphi_2}g_k} \|_1 \\
& \leq \sum_k s_k
\|V_{\varphi_1}h_k\|_2 \|V_{\varphi_2}g_k \|_2  = \|\varphi_1\|_2\, \|\varphi_2\|_2 \sum_k s_k \\
& = \|\varphi_1\|_2\, \|\varphi_2\|_2 \, \|T\|_{\mathscr{S}^1} <
\infty.
\end{align*}
Hence $\B T \in L^1(\rdd)$, and
$\B:\tc \to L^1(\rdd)$ is bounded. 

(iii) Interpolation: We  use  the preceding steps  as the endpoints
for  complex interpolation with $L^p = [L^1,L^{\infty} ]_{\theta}$ and
$\sch^p = [ \tc, \sch^{\infty} ]_{\theta} = [ \tc, B(L^2)
]_{\theta}$. The norm estimate follows from 
\begin{align*}
\|\B \|_{\sch^p \to L^p} & \le \|\B \|_{\tc \to L^1}^{\theta} \|\B \|_{\sch^{\infty} \to L^{\infty}}^{1-\theta} \\
& \le \left( \| \varphi_1 \|_2\, \| \varphi_2 \|_2 \right)^{\theta} \left( \| \varphi_1 \|_2\, \| \varphi_2 \|_2 \right)^{1 - \theta} \\
& =  \| \varphi_1 \|_2\, \| \varphi_2 \|_2.
\end{align*}
\end{proof}

Next we consider the Berezin transform with windows in the modulation
space $M^1(\rd)$ and strengthen  the preceding results. 
We first  recall an important local property of \stft s from
~\cite{cogr03} (estimate (23) in  Cor.~4.2). 
\begin{lemma} \label{localstft}
 If $\varphi _1, \varphi _2 \in M^1(\rd ) $ and $g,h \in \lrd $, then the
 product of STFTs $V_{\varphi _1} g\,  \overline{V_{\varphi _2} h }$ is in
 $M^1(\rdd )$ and satisfies the estimate
 \begin{equation}
   \label{eq:c5}
   \|V_{\varphi _1} g \overline{V_{\varphi _2} h } \|_{M^1} \leq C \|\varphi
   _1\|_{M^1}\|\varphi _2\|_{M^1} \|g\|_2 \|h\|_2 \, .
 \end{equation}
\end{lemma}

\vs

\begin{theorem}
Let $\varphi_1, \varphi_2 \in M^1(\rd)$ and $1\le p \le \infty$. Then
$\B $ is bounded from $ \sch^p$  to $  M^{p,1}(\rdd)$ and
satisfies 
\[
\|\B T \|_{M^{p,1}} \le C \, \| \varphi_1 \|_{M^1} \| \varphi_2 \|_{M^1} \| T \|_{\sch^p}
\]
for all $T\in \sch^p$, with some fixed constant $C>0$.
\end{theorem}
\begin{proof}
Again we prove the boundedness for $p=1$ and $p=\infty $ and then use
interpolation. 

(i) Case $p=1$: Assume $T\in\tc$. Using the spectral
representation of $T$,  Lemma \ref{L:berezin_representation} implies
that 
\[
\B T(z) =  \sum_{n\in\N} s_n V_{\varphi_1}h_k(z) \overline{V_{\varphi_2}g_k(z)},
\]
where $(s_n)_{n\in\N}\in \ell^1$, $s_n \ge 0$ for all $n\in\N$ and $\|T\|_{\tc} = \sum_{n\in\N} s_n$.

By Lemma~\ref{localstft} each product of STFTs $
V_{\varphi_1}h_k(z) \overline{V_{\varphi_2}g_k(z)}$ is in $M^1(\rdd )$
and satisfies 
\begin{equation*}
\| V_{\varphi_1}h_k \overline{V_{\varphi_2}g_k} \|_{M^1}  \le C \,
\underbrace{\| h_k \|_2}_{= 1} \, \|\varphi_1 \|_{M^1} \, \underbrace{\|
  g_k \|_2}_{= 1} \, \|\varphi_2 \|_{M^1} \, .
\end{equation*}
Consequently  the series for $\B T$ converges  absolutely in
$M^1(\rdd)$, and 
\begin{align*}
\| \B T \|_{M^1} & \le \sum_{n\in\N} s_n \, \| V_{\varphi_1}h_k \overline{V_{\varphi_2}g_k} \|_{M^1} \\
& \le C \, \|\varphi_1 \|_{M^1} \, \|\varphi_2 \|_{M^1}\, \sum_{n\in\N} s_n \\
& = C \, \|\varphi_1 \|_{M^1} \, \|\varphi_2 \|_{M^1} \, \|T\|_{\tc}.
\end{align*}

(ii) Case $p=\infty$: Let $T\in B(\ld)$. Since $M^1(\rd)\hookrightarrow \ld
\hookrightarrow M^{\infty}(\rd) $ with continuous embeddings, $T$  can
be considered as a bounded operator $T:M^1(\rd) \to 
M^{\infty}(\rd)$. By Lemma \ref{lem:repber}, there is a
unique $\sigma \in M^{\infty}(\rdd)$ such that 
 \begin{equation*}
\B T(z)  = \left< \sigma, T_z W( \varphi_1, \varphi_2) \right> 
 = \left( \sigma \ast W^{\ast}( \varphi_1, \varphi_2) \right)(z) \, ,
\end{equation*}
where  $F^*(z) = \overline{F(-z)}$ denotes the usual involution.\\
Since $\varphi_1, \varphi_2 \in M^1(\rd)$, the Wigner distribution $W(
\varphi_1, \varphi_2)$ is  in $ M^1(\rdd)$ by
~\cite[Thm.~12.1.2]{gr01}, 
and 
\[
\| W( \varphi_1, \varphi_2) \|_{M^1} \leq C \, \| \varphi_1 \|_{M^1}\,
\| \varphi_2 \|_{M^1}\, .
\]
Finally, we use the convolution relation for \modsp s $M^\infty \ast
M^1 \subseteq \mif $ from ~\cite{cogr03} and thus obtain 
the norm estimate
\[
\| \B T \|_{M^{\infty,1}} \le C_0 \| \sigma \|_{M^{\infty}} \, \| W(
\varphi_1, \varphi_2) \|_{M^1} \leq C \, \| \varphi_1 \|_{M^1}\, \|
\varphi_2 \|_{M^1} \, \| T \|_{B(L^2)}. 
\]

(iii) For $1< p <\infty $ we use  complex interpolation with
$ \sch^p = [\tc, B(L^2) ]_{\theta} $
and
\[ M^{p,1}(\rdd) = [M^{1}(\rdd), M^{\infty,1}(\rdd)]_{\theta} =
[M^{1,1}(\rdd), M^{\infty,1}(\rdd)]_{\theta} \, .
\]
\end{proof}

\subsection{The Berezin Transform and Localization Operators}

The connection between  the Berezin transform and \tf\ localization
operators is explained in the following statement. Again
$p'=\tfrac{p}{p-1}$ denotes the conjugate exponent of $p$.

\begin{theorem}\label{T:berezin_relevance}
Assume that  $\varphi_1, \varphi_2 \in L^2(\rd)$,  $T\in\sch^p$,
$1\le p \leq  \infty$, and $a \in L^{p^{\prime}}(\rdd)$.  Then
\begin{equation}
  \label{eq:c6}
\left< \A a, T \right> = \left< a, \B T \right>.  
\end{equation}
Here the  brackets  on the left-hand side  denote the $\sch^p$-$\sch^{p^{\prime}}$
duality, and  the  brackets on the right-hand side  denote the $L^p$-$L^{p'}$-duality.

Likewise, if  $\varphi_1, \varphi_2 \in M^1(\rd)$, $1\leq p <\infty $,
$T\in\sch^p$,
and  $a \in M^{p^{\prime}, \infty}(\rdd)$, then \eqref{eq:c6} holds
with the $M^{p',\infty }$-$M^{p,1}$-duality on the right hand side. 
\end{theorem}
\begin{proof}
Let $T \in \sch^p$ and $a\in L^{p^{\prime}}(\rdd)$ or $a\in
M^{p^{\prime},\infty}(\rdd)$, with appropriate windows.  Choose an arbitrary
orthonormal basis $(e_n)_{n\in\N}$ of $\ld$, then 
we have
\begin{align}
\left< \A a, T \right> & = \left< \loc, T \right> \notag \\
& = \mbox{trace}(T^{\ast} \loc), \notag \\
& = \sum_{n\in\N} \left< T^{\ast} \loc e_n, e_n \right>  \notag \\
& = \sum_{n\in\N} \left< \loc e_n, T e_n \right>       \label{eq:c11}
\end{align}
Since $\loc \in \sch ^{p'} $ (see Table~1) and $T\in \sch ^p $, all orthonormal
bases yield the same sum.

Now consider  the singular value decomposition of $T$,
\[
T = \sum_{k} s_k \left< \cdot , g_k \right>h_k,
\]
with $(s_k) \in \ell^p$. The orthonormal system $(g_k)_{k}$ can be
completed to an orthonormal basis $(g_n)_{n\in\N}$ of $\ld$. We
obviously have $Tg_k = s_k h_k$ for $g_k$ a member of the original
collection $(g_k)_k$, whereas $Tg_n = 0$ for all $g_n$ that were
joined to the original system $(g_k)_k$ to form a complete
basis. Thus, choosing the orthonormal basis $(g_n)_{n\in\N}$ for
$(e_n)_{n\in\N}$ in~\eqref{eq:c11}, we obtain 
\begin{align*}
\langle \cA a ,T\rangle & = \sum_{n\in\N} \left< \loc g_n, T g_n
\right> \\
&  = \sum_{k} \left< \loc g_k, T g_k \right> \\ 
&  = \sum_k s_k \left< \loc g_k, h_k \right>.
\end{align*}
Using the weak definition of $\loc $ from~\eqref{eq:c2}, i.e.,  $\left< \loc g_k, h_k \right> =
\left<a , V_{\varphi_2}h_k \, \overline{V_{\varphi_1} g_k}  \right>$
for all $k$, we find that 
\begin{align} \label{eq:c7} 
\left< \A a, T \right> & = \sum_k s_k \langle a, V_{\varphi_2}h_k \overline{ V_{\varphi_1} g_k} \rangle.
\end{align}
Now let
\[
T_N = \sum_{k=1}^N s_k \langle \cdot, g_k \rangle h_k.
\]
be the finite-rank approximation of $T$, so that  $T_N \rightarrow T$
in $\sch^p $ for $1 \le p \leq  \infty$. 

Lemma \ref{L:berezin_representation} yields the representation
\[
\B T_N = \sum_{k=1}^N s_k  \overline{ V_{\varphi_2}g_k} V_{\varphi_1} h_k.
\]
Hence
\begin{align*}
\left< a, \B T_N\right> & = \left< a, \sum_{k=1}^N s_k  V_{\varphi_2}g_k \overline{ V_{\varphi_1} h_k} \right> \\
& = \sum_{k=1}^N s_k \langle a, V_{\varphi_2}g_k \overline{ V_{\varphi_1} h_k} \rangle \\
& = \left< \A a, T_N \right>
\end{align*}
by~\eqref{eq:c7}.  Now let $N \to \infty$. Since  $\B$ is bounded  and
the  duality bracket  $\left<\cdot, \cdot \right>$ is continuous on
$\sch^{p^{\prime}}\times \sch^p$,  $L^{p^{\prime}}\times L^p$, and on
$M^{p^{\prime}, \infty}\times M^{p,1}$, we conclude
\[
\left< a, \B T\right>  = \left< \A a, T \right>
\]
for every $a\in L^{p^{\prime}}(\rdd)$ (or $a \in
M^{p^{\prime},\infty}(\rdd)$) and  $T\in \sch ^p$. 
\end{proof}

The following statement is a simple corollary of Theorem \ref{T:berezin_relevance}.

\begin{theorem}\label{T:berezin_adjoint}
Let $1 \le p < \infty$. 

(i) Assume $\varphi_1, \varphi_2 \in \ld$. Then the operator $\mathcal{A}:L^{p^{\prime}}(\rdd)  \to \sch^{p^{\prime}}$ is the Banach space adjoint of the
operator $\B:\sch^p \to L^p(\rdd)$.

(ii) Assume $\varphi_2, \varphi_2 \in M^1(\rd)$. Then the operator $\mathcal{A}:M^{p^{\prime}, \infty}(\rdd)  \to \sch^{p^{\prime}}$ is the Banach space adjoint of the
operator $\B:\sch^p \to M^{p,1}(\rdd)$.

Thus in both cases we have $\mathcal{B}^{\ast} = \mathcal{A}$.
\end{theorem}
\begin{proof}
The adjoint of $\B$ is the unique  operator $\B^{\ast}:
(L^p(\rdd))^{\ast} \to (\sch^p)^{\ast}$ and  $\B^{\ast}:
(M^{p,1}(\rdd))^{\ast} \to (\sch^p)^{\ast}$,  respectively, such that 
\[
\left<\B T, a  \right> = \left< T, \B^{\ast} a \right>
\]
for all $T\in \sch^p$ and all $a \in L^{p^{\prime}}(\rdd)$
 (or $a \in M^{p^{\prime}, \infty}(\rdd)$). Since  $(L^p)^{\ast} =
L^{p^{\prime}}$ and $(M^{p,1})^{\ast} = M^{p^{\prime}, \infty}$ and
since  by
Theorem \ref{T:berezin_relevance}, 
\[
\left<\B T, a  \right> = \left< T, \A a \right>
\]
 for all $T\in \sch^p$ and all $a \in L^{p^{\prime}}(\rdd)$
 (or $a \in M^{p^{\prime}, \infty}(\rdd)$), we conclude that 
\[ \B^{\ast} = \A.
\]
\end{proof}

For the   $L^p$-case  we obtain a symmetric statement.

\begin{corollary}\label{T:berezin_strong_adjoint}
Let $1 < p < \infty$ and $p^{\prime}$ be the conjugate exponent. 
The Berezin transform $\B:\sch^p \to L^p(\rdd)$ is the Banach space adjoint of the
operator $\mathcal{A}:L^{p^{\prime}}(\rdd)  \to \sch^{p^{\prime}}$, i.e.
$\mathcal{A}^{\ast} = \mathcal{B}$. 

For $p=1$, the operator $\B:\sch^1 \to L^1(\rdd)$ is the Banach space
adjoint of  $\mathcal{A}:L^0(\rdd)  \to \sch^{\infty}$, 
\end{corollary}
\begin{proof}
The statement is clear for $1<p<\infty $, because $L^p$ and $\sch ^p$
are reflexive Banach spaces. For $p=1$ we observe that $\cA $ maps
$L^0(\rdd )$ (the compactly supported functions in $L^\infty (\rdd )$)
into $\sch ^\infty $ (see Table~1), consequently, $\A ^* $ maps $\sch ^1 =
(\sch ^\infty )^*$ into $L^1(\rdd ) = L^0(\rdd )^*$. So \eqref{eq:c6} is still well-defined, and thus $\cA
^{\ast} = \cB $. 
\end{proof}

Note that Corollary \ref{T:berezin_strong_adjoint} does not hold for the
 case of \modsp s, since $M^{p, \infty}$ is not reflexive.

\section{The Kernel of the Berezin Transform}\label{s:kernel}

In this section we study the kernel of the Berezin transform and its
dependence of the windows $\vf _1$ and $\vf _2$. In particular, we
derive conditions  when $\cB $ is one-to-one. 

In the following every operator-valued integral is  to be understood
in the weak sense. Thus
$T = \int s(z) \pi (z) \, dz $ means that
$$
\langle Tf,g\rangle _{\rd } = \int _{\rdd } s(z) \langle \pi (z) f,
g\rangle \, dz = \langle s, V_f g\rangle _{\rdd } \, .$$
for suitable test functions $f$ and $g$. For  $s\in M^\infty (\rdd )$
and  $f,g\in M^1(\rd )$, the integral defines a bounded operator from
$M^1(\rd )$ to $M^\infty (\rd )$.  The support of a
distribution $s\in M^\infty (\rdd )$ is the smallest closed set $E$
such that $\langle s, f \rangle = 0$ for every function $f\in M^1(\rdd
)$ with compact support in the complement $E^c$. 

\begin{proposition} \label{kernelB}
Assume that  $\vf _1, \vf _2 \in M^1(\rd )$ and let $E=\{ z\in \rdd :
V_{\vf _1 } \vf _2 (z) = 0\}$. Let $s\in M^\infty (\rdd )$  and   $T: M^1(\rd ) \to M^\infty
(\rd ) $ be given by 
$$
T = \intrdd s(z) \pi (z) \, dz \, .
$$
Then the following holds: If $T$ is in the kernel of $\cB $, then $\supp
\, s \subseteq E$.

In particular, if $V_{\vf _1} \vf _2 (z) \neq 0$ for all $z\in \rdd $, then
the Berezin transform $\cB $ is one-to-one. 
\end{proposition}


\begin{proof}
The Berezin transform of $T$ is 
\begin{align*}
  \cB T(z) & = \langle T \pi (z) \vf _2 , \pi (z) \vf _1\rangle \\
&= \langle s, V_{\pi (z) \vf _1} (\pi (z) \vf _2)\rangle \\
&= \langle s, M_{\cJ z} V_{\vf _1} \vf _2\rangle  \, ,
\end{align*}
where we have used  Lemma~\ref{l:prop}(iii) in the last equality. 

If  $\cB T(z) = 0$ for all $z\in \rdd $, then 
$$
\cB T(z) = \langle s, M_{\cJ z} V_{\vf _1} \vf _2\rangle  = (s \,
\overline{V_{\vf _1} \vf _2}) \, \widehat{} \, (-\cJ z) = 0
$$
for all $z\in \rdd $. Thus the distribution $s \,
\overline{V_{\vf _1} \vf _2}$ is zero in $M^\infty (\rdd )$. 
Choose a test function $\Psi \in M^1(\rdd )$ with compact support $K
\subseteq E^c$. Since $V_{\vf _1} \vf _2 (z) \neq 0$ for $z\not \in E$,
there exists a function $\Psi _1 \in M^1(\rdd )$, such that $\Psi _1
(z) = \frac{1}{V_{\vf _1} \vf _2 (z)}$ for $z\in K$ (see~\cite{rest00}
or \cite[VIII.6.1]{kat68}). Thus we can write $\Psi $ as 
$$
\Psi = \Psi_1 \big( V_{\vf _1} \vf _2  \Psi \big) \in M^1(\rdd )
$$
on all of $\rdd $. Consequently
$$
\langle \sigma , \Psi \rangle = \langle \sigma , \Psi_1 \big(V_{\vf _1}
\vf _2   \Psi \big)\rangle = \langle \sigma \overline{V_{\vf _1}
\vf _2} , \Psi _1 \Psi \rangle = 0 \, .
$$
This is true for all $\Psi$ with compact support in $E^c$, and
therefore   $\supp s \subseteq
 E$, as claimed.  
\end{proof}

We next  investigate the injectivity of the Berezin
transform on the class of Hilbert-Schmidt operators.

\begin{theorem}\label{T:berezin_one_to_one}
Let $\varphi_1, \varphi_2 \in \ld$. Then the
Berezin transform $\B:\sch^2 \to L^2(\rdd)$ is one-to-one,  if and
only if $V_{\varphi_1} \varphi_2 (z) \not= 0$ for almost all $z \in
\rdd$. 
\end{theorem}
\begin{proof}
Note that, by formula (\ref{L:ftwig}), $V_{\varphi_1} \varphi_2 (z) \not= 0$ for almost all $z \in
\rdd$ if and only if $W(\varphi_1, \varphi_2) \, \widehat{}\, (z) \not= 0$ for almost all $z \in
\rdd$. If $T\in \hs$, there is a unique symbol  $\sigma \in \ldd$ such that
\begin{align*}
\B T(z) & = \langle \sigma, T_zW(\varphi_1, \varphi_2) \rangle \\
& = \langle \hat{\sigma}, M_{-z}W(\varphi_1, \varphi_2) \,
\widehat{}\, \,   \rangle \\
& = ( \hat{\sigma}\cdot \overline{W(\varphi_1, \varphi_2) \,
  \widehat{} \,\, } )\, \widehat{ }\, (-z).
\end{align*}
Since both $\sigma \in \ldd$ and $W(\varphi_1, \varphi_2) \in \ldd$,
the product $\hat{\sigma}\cdot \overline{W(\varphi_1, \varphi_2) \,
  \widehat{} \, }$ is in $L^1(\rdd)$ and its Fourier transform is well-defined.\\
First  assume that $W(\varphi_1, \varphi_2) \, \widehat{} \, (z) \not= 0$ for almost all $z \in
\rdd$. If $\B T \equiv 0$, then $\hat{\sigma}\cdot
\overline{W(\varphi_1, \varphi_2)\, \widehat{}\, } = 0$ almost everywhere,
whence $\hat{\sigma} = 0$ almost everywhere. Thus $\B$ is
one-to-one.\\ 
Conversely, assume that there exists a set $E \subseteq \rdd$ of
positive measure such that $W(\varphi_1, \varphi_2)\, \widehat{} \, (z) = 0$
for $z \in E$. Without loss of generality we may assume that $0 < |E|
< \infty$. Choose $\sigma \in \ldd$ as the characteristic function
$\hat{\sigma} = \chi_E$  of $E$. Then $\hat{\sigma}\cdot
W(\varphi_1, \varphi_2)\, \widehat{} \,  = 0$ on $\rdd $, hence $\B T
\equiv 0$ for the Hilbert-Schmidt operator $T\in 
\hs$ associated with symbol $\sigma$. Since $\sigma \not= 0$ and thus $T \not=
0$,  $\B$ has a non-trivial kernel and is not one-to-one. 
\end{proof}

\begin{corollary}\label{C:berezin_one_to_one_cor}
Let $\varphi_1, \varphi_2 \in \ld $. 

(i)  If $1\leq p<2$ and $V_{\varphi_1} \varphi_2 \not= 0$ almost
everywhere, then $\B$ is one-to-one from $\sch^p$ to $L^p(\rdd)$.  

(ii)  If $p> 2$ and $\B$ is one-to-one from $\sch^p$ to $L^p(\rdd)$, then $V_{\varphi_1} \varphi_2 \not= 0$ almost everywhere.
\end{corollary}
\begin{proof}
(i)  The assumption $V_{\varphi_1} \varphi_2 \not= 0$ almost
everywhere implies that $\B$ is one-to-one on $\hs$. In particular
$\B$ is one-to-one on the smaller space $\sch^p \subseteq \hs$ for
$p<2$. \\ 
(ii) If $\B$ is one-to-one on $\sch^p$ for $p>2$, the $\B$ is one-to-one on $\hs \subseteq \sch^p$ and by Theorem \ref{T:berezin_one_to_one} $V_{\varphi_1} \varphi_2 \not= 0$ almost everywhere.
\end{proof}

The following example illustrates the differences between
Proposition~\ref{kernelB} and 
Theorem~\ref{T:berezin_one_to_one}. 
Assume that
$V_{\varphi_2}\varphi_1(w) = 0$ for some $w\in \rdd$, and  consider the \tfs\  $T = \pi(w)
\in B(\ld)$. Using Lemma~\ref{l:prop}(iii),  the Berezin transform of $\pi (w)$ is given as 
\begin{align*}
\B T(z) & = \left< T\pi(z)\varphi_2, \pi(z) \varphi_1 \right> \\
& = \left< \pi(w)\pi(z)\varphi_2, \pi(z) \varphi_1 \right> \\
&= e^{-2\pi i \cJ z \cdot w} \overline{V_{\varphi_2}\varphi_1(w)}
= 0.
\end{align*}
Thus $\B T = 0 \in L^{\infty}(\rdd)$, but $T \not= 0 \in B(\ld)$, hence $\B$ is not one-to-one.

This example can be summarized as follows. 
\begin{corollary}
Let $\varphi_1, \varphi_2 \in \ld$. If $\B: B(\ld) \to L^{\infty}(\rdd)$ is one-to-one, then $V_{\varphi_1}\varphi_2(z) \not=0$ for all $z \in \rdd$. 
\end{corollary}

As a consequence of Proposition~\ref{kernelB} we will derive the
following  fact. For $T\in B(\lrd )$ let 
$$
K_T(z,w) = \langle T\pi (z)\varphi _1, \pi (w)\varphi _2\rangle
$$
be the essential kernel associated to $T$. We note that the
Berezin transform $\cB T $ is the diagonal of $K_T$. The essential
kernel appears in the description of an operator $T$ on the level of
STFTs. It is well known (see, e.g.,~\cite{cogr06}) that every  operator $T$
 can be written as 
 \begin{align*}
Tf &= (\|\vf _1\|_2 \, \| \vf _2\|_2)^{-2} \, \intrdd \intrdd  \langle f, \pi
(z) \vf _1\rangle \, \langle T\pi (z)\varphi _1, \pi (w)\varphi _2\rangle
\, \pi (w) \vf _2 \, dz dw \\
&= (\|\vf _1\|_2 \, \| \vf _2\|_2)^{-2} \,
\intrdd \intrdd V_{\vf _1}f (z) K_T(z,w) \pi (w) \vf _2 \, dz dw  
 \end{align*}
with a suitable weak interpretation of the integrals. For example, if
$\vf _1, \vf _2 \in M^1(\rd )$ and $T$ is bounded from $M^1(\rd )$ to
$M^\infty (\rd )$, then the essential kernel is continuous and
bounded, and the integral is well defined in the weak sense.

\begin{corollary}\label{corsimple}
  Assume that $\varphi _1, \varphi _2\in M^1(\rd ) $ and  $\langle \varphi _2, \pi (z) \varphi _1\rangle \neq 0$ for  all $z\in \rdd $. Then $K_T$ is uniquely determined by its diagonal $\cB T$. 
\end{corollary}

\begin{proof}
  Since $V_{\vf _1} \vf _2$ does not have any zero, the Berezin
  transform is one-to-one on the space of all bounded operators from
  $M^1$ to $M^\infty $ (Proposition~\ref{kernelB}). Since $T$ is uniquely
  determined by $\cB T$, so is its kernel $K_T$. 
\end{proof}

Statements of this type can be found in the theory of Toeplitz
operators on complex domains, see for instance ~\cite{bergercoburn94,fo89} for
Toeplitz operators on Bargmann-Fock space.  In these cases, the
statement of Corollary~\ref{corsimple} 
is immediate because the kernel is an analytic function. It is quite
surprising that such a statement can be proved  without explicit
analytic structure.  Englis~\cite{Eng06} uses the above
property as an explicit assumption in his  work on Berezin
quantization for general function spaces.


\section{Density Results}

In this section 
we investigate whether  a given  operator on $\ld$ can be approximated
by a localization operator with respect to various norms. In
particular, we would like to understand when  the set of localization
operators is dense  in $B (\ld )$ or in $\sch ^p $. 

For this we will combine Theorem~\ref{T:berezin_adjoint} and Corollary~\ref{T:berezin_strong_adjoint} with the 
following facts (cf. \cite{co90}):\\ 
Let $X,Y$ be Banach spaces and $T:X\to Y$ be a bounded operator and
let $T^*:Y^*\to X^*$ be the 
(Banach space) adjoint operator.

\begin{itemize}
\item $T^*$ is one-to-one on $Y^*$,  if and only if the range
of $T$ is norm-dense in $Y$. 
\item $T$ is one-to-one on $X$,  if and only if the range of $T^*$ is
  $w^*$-dense in $X^*$. 
\end{itemize}

Thus, in view of Theorem~\ref{T:berezin_adjoint},
  the quantization $a \to \A a $ has dense range, \fif\ 
   the Berezin
transform $\B $ is one-to-one. This property was studied  in
Section~\ref{s:kernel}.  
We now apply  these  results and derive  sufficient conditions
for the density of localization operators in a space of
operators. 

First we treat the density in $\cB (\lrd )$. 

\begin{theorem}\label{T:l_infty_weak_star_density}
Let $\varphi _1, \varphi _2\in \ld $ and  $\A:L^{\infty}(\rdd) \to B(\ld)$.\\
If $V_{\varphi_1} \varphi_2(z)
\not= 0$ for almost all $z\in\rdd$, then
$\mbox{ran}(\mathcal{A})$ is weak*-dense in $B(\ld)$.
\end{theorem}
\begin{proof}
By Corollary \ref{C:berezin_one_to_one_cor}, the condition $V_{\varphi_1} \varphi_2(z)
\not= 0$ for almost all $z\in\rdd$ implies that the Berezin transform $\B:\tc \to L^1(\rdd)$ is one-to-one.
Since $\mathcal{A} = \mathcal{B}^*$ by Theorem \ref{T:berezin_adjoint}, the injectivity of the Berezin transform $\B:\tc \to L^1(\rdd)$ is equivalent to the weak*-density of the range of $\A$.
\end{proof}

For the density of localization operators in the Schatten $p$-classes
we have the following results. As the precise statements depend on the
window class and on the range of $p$, we provide separate
formulations. 


\begin{theorem}\label{T:l_q_large_norm_density}
Assume that  $\varphi _1, \varphi _2\in \ld $ and  $ 2 \le p^{\prime} < \infty$,   and
$\A:L^{p^{\prime}}(\rdd) \to \sch^{p^{\prime}}$. 
If $V_{\varphi_1} \varphi_2(z)
\not= 0$ for almost all $z\in\rdd$, then $\mbox{ran}(\mathcal{A}) = \{
   \loc : a\in L^{p'}\}$ is norm-dense in $\sch^{p^{\prime}}$.
\end{theorem}
\begin{proof}
The proof is the same as for  Theorem~\ref{T:l_infty_weak_star_density}.
\end{proof}


\begin{theorem}\label{T:l_q_small_norm_density}
Assume that $\varphi_1, \varphi_2 \in M^1(\rd)$ and $ 1 \leq
p^{\prime} < 2$. Let $\A:L^{p^{\prime}}(\rdd) \to
\sch^{p^{\prime}}$. 
If $V_{\varphi_1} \varphi_2(z)
\not= 0$ for all $z\in\rdd$, then $\mbox{ran}(\mathcal{A})= \{
   \loc : a\in L^{p'}\}$ is norm-dense in $\sch^{p^{\prime}}$.
\end{theorem}
\begin{proof}
The proof is again identical to the proof of Theorem
\ref{T:l_infty_weak_star_density}, we use Proposition~\ref{kernelB} instead of  Theorem
\ref{T:berezin_one_to_one}. 
\end{proof}

Analogous results hold  for the  symbol class $M^{p,\infty }$, which
is significantly larger than $L^p$.

\begin{theorem}\label{T:m_infty_weak_star_density}
Let $\varphi_1, \varphi_2 \in M^1(\rd)$ and $\A:M^{p^{\prime},\infty}(\rdd) \to \sch^p$, $1 < p^{\prime} < \infty$.\\
Assume that either (a) $1 < p^{\prime} < 2$ and  $V_{\varphi_1} \varphi_2(z)
\not= 0$ for  all $z\in\rdd$, or   \\
(b) $2 \leq p^{\prime} < \infty$ and  $V_{\varphi_1} \varphi_2(z)
\not= 0$ for almost all $z\in\rdd$. 

 Then 
$\mbox{ran}(\mathcal{A})= \{
   \loc : a\in L^{p'}\}$ is norm-dense in $\sch ^{p'}$.
\end{theorem}
\begin{proof}
By Theorems~\ref{T:l_q_large_norm_density} and~\ref{T:l_q_small_norm_density} the range of $\A $ from
$L^{p'} (\rd ) $ to $\sch ^{p'} $ is norm-dense. Since
$L^{p'}(\rdd) \subseteq M^{p',\infty}(\rdd)$ for $1\leq p' \leq
\infty$,   we deduce that 
\[
\{ \loc|\;a \in L^{p'}(\rdd) \} \subseteq \{\loc|\; a\in M^{p',\infty}(\rdd) \} \subseteq \sch^{p'}
\]
is also norm-dense in $\sch ^{p'} $. 
\end{proof}

In the special case of $p^{\prime} = 2$, we obtain a  complete  
characterization for the density in terms  of the windows (as
highlighted in the introduction). 

\begin{theorem}\label{mainsimpleb}
Let $\varphi _1, \varphi _2\in \lrd $. 
Then the  following conditions are equivalent:

(i) The range  $\{A_a^{\varphi _1, \varphi _2}: a\in L^{2}(\rdd )\} $
is norm-dense in $\sch^{2}(\lrd)$.

(ii)  The Berezin transform $\mathcal{B}$ is one-to-one on $\sch^2$.

(iii)  $\langle \varphi _2 , \pi (z) \varphi _1\rangle \neq 0$  for almost all $z\in\rdd$.
\end{theorem}

\begin{proof}
The equivalence of conditions (i) and (ii) follows from
Theorem~\ref{T:berezin_adjoint}, the equivalence of conditions (ii) and (iii) is stated
in  Theorem \ref{T:berezin_one_to_one}.
\end{proof}

This result says that every Hilbert-Schmidt operator can be
approximated by localization operators. Constructive approximation
procedures and error estimates were studied in~\cite{DT10}. 

\vs

\rem\ The equivalence (i) $\, \Leftrightarrow \,$ (iii) in  
Theorem~\ref{mainsimpleb} can be proved easily and directly with the 
Weyl calculus. The Weyl symbol of the localization operator $A_a^{\vf
  _1,\vf _2}$ is given by $a \ast W(\vf _2, \vf _1)$ (see, e.g.,
\cite{wong02}). By  Theorem~\ref{T:kernel_hs} the mapping $\cA : \lrd \to
\sch ^2$ has dense range, \fif\ the subspace $L^2 \ast W(\vf _2, \vf
_1)$ is dense in $L^2(\rdd )$.  According to 
Theorem~\ref{T:tauber_l2}(i),   this is the case, \fif\ $W(\vf
  _2, \vf _1)\, \widehat{}\, (z)  \neq 0$ for almost all $z\in \rdd $. 

This simple argument cannot be extended to arbitrary Schatten classes
for two reasons: first,  we do not have an explicit description of the
Weyl symbols of operators in $\sch ^p, p\neq 2$. Second, except for
the spaces $M^1, L^1$,  and $L^2$ we  do not know how to characterize  the
translation-invariant subspace generated by a function $g$ (in our
case $g= W(\vf _2, \vf _1)$). This is a very subtle point: already for
$L^p, p\neq 1,2,\infty $, the density of $\mathrm{span}\, \{T_zg: z\in
\rdd\}$ does not only depend on the   zero set of
$\hat{g}$~\cite{LO11}. Therefore the characterization of the injectivity
of the Berezin transform on Schatten $p$-classes seem to be a
difficult problem and lies currently  beyond  our techniques.

Finally we address the failure of density of $\mathrm{ran} (\cA ) $ in
$\cB (\lrd )$ in the operator norm. For an 
intuitive understanding, we remark that  localization operators  modify the phase space content
locally, but they do not move energy  in phase
space. Therefore one may expect that localization operators  cannot 
approximate operators that describe phase space transformations, such
as the Fourier transform $\cF $  or Fourier integral operators. 
Indeed, we have the following negative result. For a comparable result
in the theory of Toeplitz operators we refer to
~\cite{bergercoburn87}.

\begin{theorem}\label{T:l_infty_counterexample_fourier_transform}
Let $\varphi _1, \varphi_2 \in \ld \setminus \{0\}$ be arbitrary and
$a \in  L^{\infty}(\rdd)$. Then 
$$
\|\cF - A_a^{\varphi _1, \varphi _2}\|_{op} \geq 1 \, . 
$$
Consequently,  the set of all localization operators with symbols in
$\li(\rdd)$ is \underline{not} dense in $\mathcal{B}(\ld)$ with
respect to the operator norm.
\end{theorem}
\begin{proof}
For $f\in\ld$ we have
\begin{align*} \|\mathcal{F}f - A_a^{\varphi_1,
\varphi_2}f\|_2^2 & =\, \|\hat{f}\|_2^2 + \|A_a^{\varphi_1,
\varphi_2}f\|_2^2 - 2\, \mbox{Re}\langle \hat{f}, A_a^{\varphi_1,
\varphi_2}f\rangle  \\
& \ge \, \|\hat{f}\|_2^2 + \|A_a^{\varphi_1, \varphi_2}f\|_2^2 -
2\,\big|\langle \hat{f}, A_a^{\varphi_1, \varphi_2}f\rangle\big| \\
& =\,\,\|\hat{f}\|_2^2 + \|A_a^{\varphi_1, \varphi_2}f\|_2^2 -
2\,\big|\langle a, \overline{V_{\varphi_1}f} V_{\varphi_2}\hat{f}\rangle\big|.
\end{align*}

Now fix  an arbitrary non-zero  $g\in\ld$. 
We will show that
\[\lim_{w \to \infty}
\big|\langle a  , \overline{V_{\varphi_1}\left( \pi (w)g \right)}
V_{\varphi_2}\widehat{\left( \pi (w) g\right)} \rangle\big| = 0\, .
\]
Let  $\eps
> 0$ be given and  choose $R>0$ such that
\[ \iint_{[-R,R]^{2d}} |V_{\varphi_1}g(z)|^2\,dz
\ge (1-\eps)\|g\|_2^2\|\varphi_1\|_2^2
\]
and
\[ \iint_{[-R,R]^{2d}} |V_{\varphi_2}\hat{g}(z)|^2\,dz
\ge (1-\eps)\|g\|_2^2\|\varphi_2\|_2^2.
\]
This  is always possible because $V_\varphi g \in \lrdd $ by the isometry
property of the STFT (Lemma~\ref{l:prop}). 
\newline 
For $z=(x,\omega ) \in \rdd$ let $\mathcal{J}z= (\omega , -x)$. Then
$\widehat{\pi (z)g} = e^{ 2\pi i x\cdot\omega}\pi
(\mathcal{J}z)\hat{g}$.  

Now  define $U_{\eps}:= \mathcal{J}w + [-R,R]^{2d}$ and $V_{\eps}:= w
+ [-R,R]^{2d}$. By the covariance property ~\eqref{l:prop}(i)  and  a change of variables we
obtain 
\begin{align*}
\iint_{V_{\eps}}
|V_{\varphi_1}(\pi(w) g)(z)|^2\,dz & =  \iint_{w+ [-R,R]^{2d}}
|V_{\varphi_1}g(z-w)|^2\,dz \\
&= \iint_{[-R,R]^{2d}} |V_{\varphi_1}g(z)|^2\,dz \\
& \ge (1-\eps)\|g\|_2^2\|\varphi_2\|_2^2 \\
& = (1-\eps)\|\pi(w) g\|_2^2\|\varphi_2\|_2^2\, ,
\end{align*}
and similarly
$$
\iint _{U_\epsilon } |V_{\varphi_2}\widehat{(\pi(w) g)}(z)|^2\,dz
 = \iint_{[-R,R]^{2d}}
|V_{\varphi_2}\hat{g}(z)|^2\,dz \geq  (1-\eps)\|\pi(w)
g\|_2^2\|\varphi_1\|_2^2 \, .
$$
By choosing $|w|$ large enough,  we can achieve
$U_{\eps} \cap V_{\eps} = (\mathcal{J}w + [-R,R]^{2d}) \cap (w+
[-R,R]^{2d}) = \emptyset$, so that 
$U_{\eps} \subset \mathbb{R}^{2d}\setminus V_{\eps}$
and $V_{\eps} \subset \mathbb{R}^{2d}\setminus
U_{\eps}$. Then
\begin{equation*}
   \iint_{U_{\eps}}
|V_{\varphi_1}(\pi(w)g)(z)|^2\,dz  \le
\iint_{\mathbb{R}^{2d}\setminus V_{\eps}}
|V_{\varphi_1}(\pi(w) g)(z)|^2\,dz
 \le \eps \|\pi(w) g\|_2^2\|\varphi_1\|_2^2
\end{equation*}
and
\begin{equation*}
 \iint_{V_{\eps}}
|V_{\varphi_2}(\widehat{\pi(w) g})(z)|^2\,dz
 \le \iint_{\mathbb{R}^{2d}\setminus U_{\eps}}
|V_{\varphi_2}(\widehat{\pi(w) g})(z)|^2\,dz
 \le \eps \|\pi(w) g\|_2^2\|\varphi_2\|_2^2.
 \end{equation*}

Writing $f=\pi(w)g $, we conclude that 
{\allowdisplaybreaks
\begin{align*} \big|\langle a &
,  \overline{V_{\varphi_1}f} \, V_{\varphi_2}\hat{f} \rangle\big| = \big|
\iint_{\mathbb{R}^{2d}}
a(z)\overline{V_{\varphi_1}f(z)}V_{\varphi_2}\hat{f}(z)\,dz
\big|
\\
& \le \|a\|_{\infty}\iint_{\mathbb{R}^{2d}}
|V_{\varphi_1}f(z)|\,|V_{\varphi_2}\hat{f}(z)|\,dz
\\
&= \|a\|_{\infty}\bigg[\iint_{U_{\eps}}
|V_{\varphi_1}f(z) |\,|V_{\varphi_2}\hat{f}(z) |\,dz + \iint_{\mathbb{R}^{2d}\setminus U_{\eps}}
 |V_{\varphi_1}f(z) |\,|V_{\varphi_2}\hat{f}(z) | \,dz\bigg]
\\
& \le\|a\|_{\infty}\bigg[\left(\iint_{U_{\eps}}
|V_{\varphi_1}f(z)|^2\,dz\right)^{1/2}\left(\iint_{U_{\eps}}
|V_{\varphi_2}\hat{f}(z) |^2\,dz\right)^{1/2} \\
& \quad\quad\quad\quad + \left(\iint_{\mathbb{R}^{2d}\setminus
U_{\eps}}
|V_{\varphi_1}f(z)|^2\,dz\right)^{1/2}\left(\iint_{\mathbb{R}^{2d}\setminus
U_{\eps}}
|V_{\varphi_2}\hat{f}(z) |^2\,dz\right)^{1/2}\bigg]\\
& \le\|a\|_{\infty}\bigg[\ \left(\iint_{U_{\eps}}
|V_{\varphi_1}f(z)|^2\,dz\right)^{1/2}|V_{\varphi_2}\hat{f}\|_2  \\
&\quad\quad\quad\quad + \|V_{\varphi_1}f\|_2 \left(\iint_{\mathbb{R}^{2d}\setminus
U_{\eps}}
|V_{\varphi_2}\hat{f}(z) |^2\,dz\right)^{1/2}  \bigg]\\
& \le\|a\|_{\infty}\Big[\|f\|_2 \,
\|\varphi_1\|_2 \cdot\sqrt{\eps}\,\|f\|_2 \,\|\varphi_2\|_2 +
\sqrt{\eps}\,\|f\|_2 \,\|\varphi_1\|_2 \cdot \|f\|_2 \,
\|\varphi_2\|_2 \Big] \\
& = 2\,\|a\|_{\infty}\,\|\varphi_1\|_2 \,\|\varphi_2\|_2 \,\|f\|^2_2 \cdot
\sqrt{\eps} \\
& = C\,\|g\|^2_2  \cdot \sqrt{\eps}
\end{align*}}
This estimate holds for all sufficiently large $w\in \rdd$, and  the
constant is  
$C=2\,\|a\|_{\infty}\,\|\varphi_1\|_2\|\varphi_2\|_2$ independently  of
$w\in \rdd$. Thus the claim
\[\lim_{w \to \infty}
\big|\langle a  , \overline{V_{\varphi_1}\left( \pi (w) g\right)} V_{\varphi_2}\widehat{\left( \pi (w)g \right)}
 \rangle\big| = 0
\]
is proved.
\newline For nonzero $g\in\ld$ and $\epsilon >0$  choose $w\in \rdd  $ such that
\[\big|\langle a \, V_{\varphi_1}(\pi(w) g), V_{\varphi_2}(\widehat{\pi(w)g})
\rangle\big| \le
\frac{\epsilon}{2} \,  \|g\|_2^2.
\]
Then with $f :=\pi(w) g$ we have
\begin{align*} \|\mathcal{F}f - A_a^{\varphi_1,
\varphi_2}f\|_2^2 & \ge \,\,\|\hat{f}\|_2^2 + \|A_a^{\varphi_1,
\varphi_2}f\|_2^2 - 2\,\big|\langle V_{\varphi_2}\hat{f},
aV_{\varphi_1}f\rangle\big|\\
&\ge (1-\epsilon ) \|g\|_2^2  = (1-\epsilon ) \|f\|_2^2 \, .
\end{align*}
since $\|f\|_2 = \|\hat{f}\|_2 = \|g\|_2 = \|\hat{g}\|_2$. This shows
\[\|\mathcal{F} - A_a^{\varphi_1, \varphi_2}\|_{B(L^2)} \ge
1
\]
for all $a\in L^{\infty}(\mathbb{R}^{2d})$, which was to be proved.
\end{proof}

 A similar  argument applies to arbitrary  operators of the
metaplectic representation, since these  operators correspond to
linear transformations of 
phase space. By a more sophisticated argument one can also show that the Fourier
transform cannot be approximated by a  localization operator $\loc $  with a
distributional symbol $a\in M^\infty $ instead of $a\in L^\infty (\rdd
)$.


\def\cprime{$'$} \def\cprime{$'$} \def\cprime{$'$} \def\cprime{$'$}
  \def\cprime{$'$} \def\cprime{$'$}

\bibliographystyle{abbrv}
\bibliography{bibliography-dominik,new,general}

\end{document}